\documentclass{amsart}
\usepackage {amsmath, amscd}
\usepackage {amssymb}
\usepackage{hyperref}

\usepackage{mathpazo}

\advance\voffset by -1.5cm
\advance\hoffset by -1.5cm
\def\<{\langle}
\def\>{\rangle}

\numberwithin{equation}{section}

\def\AA{{\mathcal A}}

\def\DD{{\mathcal D}}

\def\FF{{\mathcal F}}
\def\GG{{\mathcal G}}
\def\HH{{\mathcal H}}

\def\LL{{\mathcal L}}

\def\bbR{\mathbb{R}}
\def\bbC{\mathbb{C}}

\def\bbN{\mathbb{N}}

\def\bbH{\mathbb{H}}
\def\bbE{\mathbb{E}}

\newcommand{\tr}{\mathop{\rm Tr}}

\newcommand{\diag}{\mathop{\rm diag}}

\newtheorem{lemma}{Lemma}[section]
\newtheorem{proposition}[lemma]{Proposition}%[section]
\newtheorem{theorem}[lemma]{Theorem}%[section]
\newtheorem{corollary}[lemma]{Corollary}

\newtheorem*{KS}{Kadison--Singer Conjecture (KS)}
\newtheorem*{PC}{Paving Conjecture (PC)}

\theoremstyle{definition}
\newtheorem{remark}[lemma]{Remark}
\newtheorem{definition}[lemma]{Definition}
\newtheorem{example}[lemma]{Example}
\newtheorem{exer}[lemma]{Exercise}

\let\oldtocsection=\tocsection

\let\oldtocsubsection=\tocsubsection

\renewcommand{\tocsection}[2]{\hspace{0em}\oldtocsection{#1}{#2}}
\renewcommand{\tocsubsection}[2]{\hspace{1em}\oldtocsubsection{#1}{#2}}
{\parindent=1.5cm\bigskip\begin{minipage}{13cm}\begin{exer}\begin{small}}%
{\end{small}\end{exer}\end{minipage}\bigskip}

\title[The solution to the Kadison--Singer Problem]{The solution to the Kadison--Singer Problem: yet another presentation}

\author{Dan Timotin}

\begin{document}

\begin{abstract}
	In the summer of 2013 Marcus, Spielman, and Srivastava gave a surprising and beautiful solution to the Kadison--Singer problem. The current presentation  is slightly more didactical than other versions that have appeared since; it hopes to contribute to a thorough understanding of this amazing proof.
\end{abstract}

\maketitle

\tableofcontents

\section{Introduction}

The Kadison--Singer Problem has been posed in~\cite{KS} in the fifties, probably in relation to a statement of Dirac concerning the foundations of quantum mechanics. It has soon acquired a life of its own. On one hand, there have been several notable attempts to prove it. On the other hand, it has been shown that it is equivalent to various problems in Hilbert space theory, frame theory, geometry of Banach spaces, etc. However, for five decades the problem has remained unsolved.

It is therefore very remarkable that in 2013 a proof has been given by Marcus, Spelman and Srivastava in~\cite{MSS}. The methods used were rather unexpected; moreover, they had shown their strength in some totally unrelated areas (Ramanujan graphs). They  also have a very elementary flavour: most of the proof is based on a delicate analysis of the behavior of polynomials in one or several variables.

In the year and a half that has passed a better grasp of the proof has been achieved, most notably through Terence Tao's entry in his blog~\cite{T} (but see also~\cite{MSS3}). It still remains an astonishing piece of research, obtaining  spectacular results on a long standing conjecture through some not very complicated and apparently unrelated arguments.

The purpose of these notes is to contribute towards a better understanding of the MSS proof. There is of course no pretense to any originality: the content is essentially in~\cite{MSS}, with some supplementary simplification due to~\cite{T} (and occasionally to~\cite{V}). But we have tried to make it more easy to follow by separating clearly the different steps and emphasizing the main arguments; also, in various places we have gone into more details than in the other presentations. It is to be expected that the methods of~\cite{MSS}  might lead to new fruitful applications, and so it seemed worth to analyze them in detail.

It is clear from the above that the notes concentrate on the MMS proof, so there will be very little about the Kadison--Singer problem itself and about the plethora of research that had evolved in the last fifty years on its relations to other domains. In particular, with one exception that we need to use (the paving conjecture), we will not discuss the different reformulations and equivalent statements that have been obtained. For all these matters, one may consult former beautiful presentations, as for instance~\cite{Ca}. 

We will give in the next section a brief presentation of the original problem, as well as of another assertion, the paving conjecture, which has been shown soon afterwards to imply it. The description of the remaining part of the paper is postponed to Section~\ref{se:intermezzo}, where the reader will have a general overview of the development of the proof.

These notes have been written for a series of lectures given in December 2014 at the Indian Statistical Institute in Bangalore, in the framework of the meeting \emph{Recent Advances in Operator Theory and Operator Algebras}. We thank B.V.R. Bath, J. Sarkar, and V.S. Sunder for the excellent work done in organizing the workshop and the conference, as well as for the invitation to present the lectures. 

\section{The Kadison--Singer problem}\label{se:KSproblem}

\subsection{Pure states}

The material in this subsection is contained in standard books on $C^*$
-algebras (see, for instance,~\cite{KR}).

We  denote by $B(\HH)$ the algebra of all bounded linear operators on the Hilbert space $\HH$.
A \emph{$C^*$-algebra} $\AA\subset B(\HH)$ is a norm closed subalgebra of $B(\HH)$, closed to the operation of taking the adjoint, and containing the identity.

A \emph{state} on a $C^*$-algebra $\AA$ is a linear continuous map $\phi:\AA\to\bbC$, which is positive (meaning that $\phi(a^*a)\ge 0$ for any $a\in\AA$), and such that $\phi(I)=1$. One proves then that $\|\phi\|=1$ and that $\phi$ satisfies the \emph{Cauchy--Schwarz} inequality
\begin{equation}\label{eq:Cauchy-Schwarz for states}
|\phi(b^*a)|^2 \le \phi(a^*a)\phi(b^*b)
\end{equation}
for all $a,b\in \AA$.

The set $\mathfrak{S}(\AA)$ of all states on $\AA$ is a convex, $w^*$-compact subset of $\AA^*$. A state $\phi$ is called \emph{pure} if it is an extreme point of $\mathfrak{S}(\AA)$.

\begin{example}\label{ex:commutative}
	If $\AA$ is commutative, then by Gelfand's Theorem it is isomorphic to $C(X)$, the algebra of continuous functions on the compact space $X$ of all \emph{characters} (multiplicative homomorphisms) $\chi:\AA\to \bbC$. The dual $C(X)^*$ is formed by all Borel measures on $X$, and $\mathfrak{S}(C(X))$ is the set of probability measures on $X$. Pure states are precisely Dirac measures. In particular (and this is a fact that we will use below) a pure state on a commutative $C^*$-algebra is multiplicative.
	\end{example}

	\begin{example}\label{ex:pure states on B(H)}
		 If $\AA=B(\HH)$, $\xi\in \HH$, and $\|\xi\|=1$, then one can prove that $\phi_\xi(T):=\< T\xi, \xi\>$ is a pure state. This fact will not be used in the sequel.
	\end{example}

By a theorem of Krein (see, for instance,~\cite[Ch.I.10]{Na}) any state $\phi$ on a $C^*$-algebra $\AA$ extends to a state $\tilde{\phi}$ on $B(\HH)$. The set $K_\phi$ of all extensions of $\phi$ is a convex $w^*$-compact subset of $B(\HH)^*$.

\begin{lemma}\label{le:pure state extension}
	If $\phi$ is a pure state on $\AA\subset B(\HH)$, then the extreme points of $K_\phi$ are pure states of $B(\HH)$.
\end{lemma}

\begin{proof}
	Suppose $\phi$ is an extreme point of $K_\phi$. If $\tilde{\phi}=\frac{1}{2}(\psi_1+\psi_2)$, with $\psi_1,\psi_2$ states on $B(\HH)$, then $\phi=\frac{1}{2}(\psi_1|\AA+\psi_2|\AA)$. Since $\phi$ is pure, we must have $\psi_1|\AA=\psi_2|\AA=\phi$, so $\psi_1, \psi_2\in K_\phi$, and therefore $\psi_1= \psi_2=\phi$.
\end{proof}

Consequently, a pure state $\phi$ on $\AA$  has a unique extension to a state on $B(\HH)$ if and only if it has a unique \emph{pure} extension to a state  on $B(\HH)$.

\subsection{The Kadison--Singer conjecture}

From now on we will suppose that the Hilbert space $\HH$ is $\ell^2=\ell^2(\bbN)$ and we will consider matrix representations of operators on $B(\ell^2)$ with respect to the usual canonical basis of $\ell^2$. We define $\DD$ to be the $C^*$-algebra of  operators on $\ell^2$ whose matrix is diagonal.  Note that the map $\diag:B(\ell^2)\to \DD$ which sends an operator $T$ to the diagonal operator having the same diagonal entries is continuous, positive, of norm~1.

We may now state the \emph{Kadison--Singer Problem:}

\begin{quote}
	\emph{Does any pure state on $\DD$ extend uniquely to a  state on $B(\ell^2)$?}
\end{quote}

Although Kadison and Singer originally thought a negative answer to this question as more probable, in view of its eventual positive answer we will state the conjecture in the affirmative form.

\begin{KS}
	Any pure state on $\DD$ extends uniquely to a  state on $B(\ell^2)$.
\end{KS}

The first thing to note is that any state $\phi\in \mathfrak{S}(\DD)$ has a ``canonical''
 extension to $\mathfrak{S}(B(\ell^2))$, given by
 \begin{equation}\label{eq:canonical tilde phi}
 \tilde{\phi} (T)=\phi(\diag(T)).
 \end{equation}
 So the problem becomes whether $\tilde{\phi}$ is or not the unique extension of $\phi$ to $B(\ell^2)$. If $\psi$ is another extension of $\phi$ and $T\in B(\ell^2)$, then 
 \[
 \psi(T-\diag T)= \psi(T)-\phi(\diag T)= \psi(T)-\tilde{\phi}(T).
 \]
 So $\psi=\tilde{\phi}$ if and only if $\psi(T-\diag T)=0$ for any $T\in B(\ell^2)$, which is equivalent to say that $\psi(T)=0$ for any $T\in B(\ell^2)$ with $\diag T=0$. As a consequence, we have the following simple lemma:
 
 \begin{lemma}\label{le:simple equivalence}
 	(KS) is true if and only if  any extension $\psi\in \mathfrak{S}(B(\ell^2))$ of a pure state on $\AA$ satisfies 
 	\[
 	\diag T=0\implies\psi(T)=0.
 	\]
 \end{lemma}

In fact, pure states of $\DD$ can be described more precisely. Indeed, being a commutative algebra, $\DD$ is isomorphic to $C(X)$ (as noted in Example~\ref{ex:commutative}). One can  identify $X$ precisely: it is $\beta \bbN$, the Stone-Cech compactification of $\bbN$. We do not need this fact, but will use only a simple observation.

\begin{lemma}\label{le:pure state on projection}
	If $\phi$ is a pure state on $\DD$ and $P\in\DD$ is a projection, then $\phi(P)$ is either 0 or 1.
\end{lemma}

\begin{proof}
	It has been noted above (see Example~\ref{ex:commutative}) that $\phi$ is multiplicative. Then $\phi(P)=\phi(P^2)=\phi(P)^2$, whence $\phi(P)$ is either 0 or 1.
\end{proof}

\begin{remark}\label{re:dirac}
	As hinted in the introduction, although in the original paper~\cite{KS} there is no mention of quantum mechanics, in subsequent papers the authors state as source for the problem the work of Dirac on the foundation of quantum mechanics~\cite{Di}. For some comments on this, see Subsection~\ref{se:final remarks}.1 below.
\end{remark}

\subsection{The Paving Conjecture}\label{sse:PC}

Instead of dealing directly with the Kadison--Singer conjecture, we intend to prove a statement about finite dimensional matrices, which is usually known as Anderson's \emph{paving conjecture}~\cite{A}. We use the notation $\DD_m$ to indicate diagonal $m\times m$ matrices and $\diag_m$ the corresponding map from $M_m(\bbC)$ to $\DD_m$. 

\begin{PC}
	For any $\epsilon>0$ there exists  $r\in\bbN$ such that the following is true: 
	
	For any $m\in \bbN$ and $T\in B(\bbC^m) $ with $\diag_m T=0$,  there exist projections $Q_1, \dots, Q_r\in \DD_m$, with $\sum_{i=1}^{r} Q_i=I_m$, and
	\[
	\|Q_i TQ_i\|\le \epsilon \|T\|
	\] 
	for all $i=1, \dots, r$.
\end{PC}

A diagonal projection $Q\in \DD_m$ has its entries 1 or 0, so it is defined by a subset $S\subset \{1, \dots, m \}$. Thus diagonal projections $Q_1, \dots Q_r\in \DD_m$ with $\sum_{i=1}^{r} Q_i=I_m$ correspond to partitions $\{1, \dots, m \}=S_1\cup\cdots\cup S_r$, $S_i\cap S_j=\emptyset$ for $i\not=j$.

It is important that in the statement of (PC) the number $r$ does not depend on $m$. This allows us to deduce from (PC)  a similar statement, in which $\bbC^m$ is replaced with the whole $\ell^2$, is also true. We formulate this as a lemma.

\begin{lemma}\label{le:PC implies PC infinity}
	If (PC) is true, then for any $\epsilon>0$ there exists $r\in \bbN$ such that, for any $T\in B(\ell^2)$ with $\diag T=0$ one can find projections $Q_1, \dots, Q_r\in \DD$, with $\sum_{i=1}^{r} Q_i=I$, and
	\[
	\|Q_i TQ_i\|\le \epsilon \|T\|
	\] 
	for all $i=1, \dots, r$.
\end{lemma}

 \begin{proof}
 	Embed $\bbC^m$ canonically into $\ell^2$ on the first $m$ coordinates and denote by $E_m$ the corresponding orthogonal projection. For $T\in B(\ell^2)$ denote $T_m=E_mTE_m$. Applying (PC), one finds diagonal projections $Q^{(m)}_1, \dots, Q^{(m)}_r$, such that $\sum_{i=1}^{r} Q^{(m)}_i=I_m$ and $\|Q^{(m)}_iT_mQ^{(m)}_i\| \le \epsilon\|T_m\|$.
 	
 	Now, diagonal projections in $B(\ell^2)$ can be identified with subsets of $\bbN$, and therefore with elements in the compact space $\{ 0, 1\}^\bbN$. In this compact space any sequence has a convergent subsequence; therefore a diagonal argument will produce an increasing subsequence of positive integers $m_k$, such that for each $i=1, \dots, r$ we have $Q^{(m_k)}_i\to Q_i$ for some $Q_i$. We have 
 	\[
 	\sum_{i=1}^{r} Q_i =\lim_{k\to\infty} \sum_{i=1}^{r} Q^{(m_k)}_i= \lim_{k\to\infty} I_{m_k}=I.
 	\]
 	If $\xi, \eta\in \ell^2$ are vectors with finite support, then $\xi, \eta\in \bbC^{d_k}$ for some $k$, and then
 	\[
 	\begin{split}
 		|\< Q_iTQ_i \xi, \eta\>|& = 	|\< TQ_i \xi, Q_i \eta\>|
 		=	|\< T_{m_k}Q^{(m_k)}_i \xi,Q^{(m_k)}_i  \eta\>|
 		= 	|\<Q^{(m_k)}_i  T_{m_k}Q^{(m_k)}_i \xi, \eta\>|\\
& 		\le \| Q^{(m_k)}_i  T_{m_k}Q^{(m_k)}_i\| \|\xi\| \|\eta\|
 		\le \epsilon \|T\|  \|\xi\| \|\eta\|.\qedhere
 	\end{split}
 	\]
 \end{proof}

The Paving Conjecture is actually equivalent to the Kadison--Singer Conjecture, but we will need (and prove) only one of the implications.

\begin{proposition}\label{pr:PC implies KS}
	The Paving Conjecture implies the Kadison--Singer Conjecture.
\end{proposition}

\begin{proof}
	Fix $\epsilon>0$, and suppose that $r$ satisfies the conclusion of Lemma~\ref{le:PC implies PC infinity}. Take a pure state $\psi\in \mathfrak{S}(B(\ell^2))$ and an operator $T\in B(\ell^2)$ with $\diag T=0$. By Lemma~\ref{le:simple equivalence} we have to show that $\psi(T)=0$. 
	
	Let $Q_i$ be the diagonal projections associated to $T$ by Lemma~\ref{le:PC implies PC infinity}. By Lemma~\ref{le:pure state on projection}, $\psi(Q_i)=\phi(Q_i)$ is 0 or 1 for each $i$. Since $1=\phi(I)= \sum_{i=1}^{r}\phi(Q_i)$, it follows that there exists some $i_0$ for which $\phi(Q_{i_0})=1$, while $\phi(Q_i)=0$ for $i\not= i_0$.
	
	If $i\not= i_0$, then~\eqref{eq:Cauchy-Schwarz for states} implies
	\[
	|\psi(Q_i R)|\le \psi(Q_i^* Q_i)\psi (R^*R) = \psi(Q_i)  \psi (R^*R)=0,
	\]
	and similarly $\psi (RQ_i)=0$ for all $R\in B(\ell^2)$. Therefore
	\[
	\psi(T)= \sum_{i=1}^{r}\sum_{j=1}^{r} \psi(Q_i T Q_j)= \psi(Q_{i_0}TQ_{i_0}).
	\] 
	But the  projections $Q_i$ have been chosen such as to have $\|Q_{i_0}T Q_{i_0}\|\le \epsilon\|T\|$, so
	\[
	|\psi(T)|\le \|Q_{i_0}T Q_{i_0}\|\le \epsilon\|T\|.
	\]
	Since this is true for any $\epsilon>0$, it follows that $\psi(T)=0$, and the proposition is proved.
\end{proof}

%It is worth to note that the Paving Conjecture is equivalent to an apparently stronger statement, which does not necessarily imply that $\diag_N T=0$. This is:
%
%\begin{MPC}
%	For any $\epsilon>0$ there exists  $r\in\bbN$ such that, if $N\in \bbN$ and $T\in B(\bbC^N) $,  then we may find projections $Q_1, \dots Q_r\in \DD_N$, with $\sum_{i=1}^{r} Q_i=I_N$, and
%	\[
%	\|Q_i TQ_i\|\le \epsilon \|T\|+\|\diag T\|.
%	\] 
%	for all $i=1, \dots, r$.
%\end{MPC}
%
%Indeed, (MPC)$\implies$(PC) obviously. For the converse, take $T\in B(\bbC^N) $, and choose $Q_1, \dots, Q_r$ as given by (PC) for $\epsilon/2$ and $T-\diag T$; then
%\[
%\|Q_i TQ_i\|\le \|Q_i (T-\diag T) Q_i\|+\|Q_i \diag T Q_i\|\le \epsilon/2 \|T-\diag T\|+\|\diag T\|\le \epsilon \|T\|+\|\diag T\|.
%\]
%

%
%
%
%
\section{Intermezzo: what we will do next and why}\label{se:intermezzo}

\subsection{General plan}

As noted above, we intend to prove the Paving Conjecture. The proof will lead us on an unexpected path, so we   explain here its main steps.

The Paving Conjecture asks us to find, for a given matrix $T$,  diagonal projections $Q_i$ that achieve certain norm estimates (namely, $\|Q_iTQ_i\|\le \epsilon\|T\|$). Among the different ways to estimate the norm, the proof in~\cite{MSS} choses a rather unusual one: it uses the fact that the norm of a positive operator is its largest eigenvalue. So we have to consider  characteristic polynomials of  matrices---in fact, the largest part of the proof is dedicated to estimating roots of such polynomials. (Although it has nothing to do with (KS), one should note the added benefit that we find a way to control with no extra effort all eigenvalues of the matrix, not only the largest one.)

On the other hand, to achieve this control we need to make an unexpected detour: though the characteristic polynomial depends on a single variable, in order to control it one has to go through multivariable polynomials and to use the theory of real stable polynomials as developed by Borcea and Br\"and\'en~\cite{BB}. This may seem unnatural, but it should be mentioned that Borcea and Br\"and\'en have already obtained through their methods spectacular results, in particular solving long-standing conjectures in matrix theory that also seemed at first sight to involve just a single complex variable~\cite{BB, BB2}. So maybe one should not be so surprised after all.

A second feature of the proof is its use, at some point, of a random space. After obtaining certain results about eigenvalues of usual matrices, suddenly random matrices appear on the scene. In fact, the use of randomness is not really essential; it rather provides a convenient notation for computing averages. As noted in the previous section, to prove (PC) we need to find a partition  of a finite set $\{1,\dots, m \}$ into $r$ subsets with certain properties.
The random space eventually considered is finite; its elements are all different such partitions, and no subtle probability is used: all decompositions are assumed to be equally probable. What we will achieve eventually is an estimate on the average of the largest eigenvalue, which will lead to an individual estimate for at least one point of the random space---that is, for one partition. This will be the desired partition.

\subsection{Sketch of the proof}

We summarize here the development of the proof. As announced above, we intend to discuss the eigenvalues of positive matrices, which are roots of the characteristic polynomial. So we need some preparation concerning polynomials and their roots; this is done first in one variable in Section~\ref{se:univ poly}. The main result here is Theorem~\ref{th:nice families}, that shows that certain families of polynomials have roots that behave  unexpectedly well with respect to averages. This will be used in Section~\ref{se:random} to link eigenvalues of random matrices to their averages.

But we have to go to polynomials in several variables, namely real stable polynomials, which are defined by a condition on their roots.  Section~\ref{se:real stable} is dedicated to real stable polynomials; after presenting their main properties, we are especially interested in some delicate estimate on the location of the roots, which is done through an associated function called the barrier function. The properties of the barrier function represent the most technical and not very transparent part of the proof. The main thing to be used in the sequel is Theorem~\ref{th:real stable no zeros after 1-partial}, that puts some restriction on the roots of a real stable polynomial.

We apply these facts to characteristic polynomials in Section~\ref{se:mixed char}. The voyage through several variables done for polynomials has a correspondent here in the introduction of the \emph{mixed characteristic polynomial}, which depends on several  matrices. It happens to be the restriction to one variable of a real stable polynomial, and so Theorem~\ref{th:real stable no zeros after 1-partial} can be used in Theorem~\ref{th:sum=1, trace small} to bound  the roots of a  mixed characteristic polynomial. Further, this bound translates in a bound for a usual characteristic polynomial 
 in the particular case when the matrices have rank one, since then  the mixed characteristic polynomial is precisely the characteristic polynomial of their sum.

Section~\ref{se:random} introduces random matrices; as discussed above, the probability space in view is that of all possible partitions. The main result, Theorem~\ref{th:min root le mixed char of expectation}, uses the results of Section~\ref{se:univ poly} to show that for a sum of independent random matrices of rank one,  the  eigenvalues of its average yield estimates for the averages of its eigenvalues, and thus for the eigenvalues of at least  one point of the probability space. In particular, applying this fact in conjunction with the bound on eigenvalues obtained in Section~\ref{se:mixed char}, we will obtain a partition with certain norm properties in Theorem~\ref{th:first theorem with partitions}.

Finally, this last fact is put to good use in Section~\ref{se:main proof} to obtain a proof of the  Paving Conjecture. The first step, that uses Theorem~\ref{th:first theorem with partitions}, obtains for orthogonal projections a quantitative version of (PC). To go from projections to general operators is well known since several decades and may be done in different ways. Here we use a dilation argument taken from~\cite{V} to obtain the Paving Conjecture for selfadjoint matrices; going to general matrices is then immediate.

\section{Analytic functions and univariate polynomials}\label{se:univ poly}

\subsection{Preliminaries}

The next theorem in complex function theory is a consequence of Cauchy's argument principle.

\begin{theorem}\label{th:Cauchy}
 Suppose $(f_n)$ is a sequence of analytic functions on a domain $D\subset \bbC$, which converges uniformly on compacts to the function $f\not\equiv 0$. If $\Gamma$ is a simple contour contained in $D$ such that $f$ has no zeros on $\Gamma$, then there is $n_0\in\bbN$ such for $n\ge n_0$ the number of zeros of $f_n$ and of $f$ in the interior of $\Gamma$  coincide.
		\end{theorem}

	The next corollary is usually called Hurwitz's Theorem if $m=1$. The general case follows simply by induction (exercise!).
		
		\begin{corollary}\label{co:Herglotz and other's}
			Suppose $p_n(z_1, \dots, z_m)$ are polynomials in $m$ variables, such that $p_n\to p$ uniformly on compacts in some domain $D\subset \bbC^m$. If $p_m$ has no zeros in $D$ for all $m$, then  either $p$ is identically zero, or it has no zeros in $D$.
		\end{corollary}

%\begin{proof}
%	For $m=1$ the result follows immediately from Theorem~\ref{th:Cauchy}. 
%	
%	Assume then it is true up to $m-1$ and we want to prove it for $m$. Suppose that $p(w_1,\dots, w_m)=0$. Let $D'\subset\bbC^{m-1}$ be defined by $D'=\{ z'\in\bbC^{m-1}:(z',w_m)\in D \}$
%	
%	
%	
%	Consider the polynomials in $m-1$ variables $q_n, q$ defined in $D'$ by $q_n(z')=p_n(z', w_m)$, $q(z')=p(z', w_m)$.
% Then $q_n$ has no zeros in $D'$ and $q_n$ converges on compacts to $q$. Since $q(w_1, \dots, w_{m-1})=0$, the induction hypothesis implies that $q$ is identically zero.
% 
% Now for any fixed $w\in D'$ we may apply the case $m=1$ to the one variable polynomials $p_n(w, \cdot)$ and $p(w, \cdot)$ defined on a neighborhood of $w_n$. Again we obtain that the limit has to be identically zero. So $p(w, z_m)=0$ for all $w\in D'$ and $z_m$ in a neighborhood of $w_m$, which implies by analytic continuation that $p\equiv0$.
%\end{proof}

If $ f $ is a polynomial of degree $n$ with all coefficients and all   roots real, we denote its roots by
\[
\rho_n(f)\le\dots\le \rho_1(f).
\]

\begin{corollary}\label{co:continuity of real roots}
	Suppose $p_s(z)=\sum_{i=1}^{n} a_i(s)z^i$, with $a_i:I\to\bbR$ continuous functions on an interval $I\subset \bbR$, $a_n(s)\not=0$ on $I$. If $p_s$ has real roots for all $s\in I$, then the roots $\rho_1(p_s), \dots, \rho_n(p_s)$ are continuous functions of $s\in I$.
\end{corollary}

\begin{proof} We use induction with respect to $n$. The case $n=1$ is obvious. Then, for a general $n$,
	we prove first that $\rho_1(p_s)$ is continuous, say in 
	 $s_0\in I$.
	 Take $\epsilon>0$, and suppose also that
	  $p_{s_0}(s_0\pm\epsilon)\not=0$. 
	By continuity of $a_i$,  $p_s(s_0\pm\epsilon)\not=$ for $s$ sufficiently close to $s_0$, and so $p_s(z)\not=0$
	for $z$ on the  circle $\Gamma$ of diameter $[s_0-\epsilon, s_0+\epsilon]$ (since all $p_s$ have real roots). By Theorem~\ref{th:Cauchy} all $p_s$ have at least one root inside $\Gamma$ for $s$ sufficiently close to $s_0$. A similar argument, using a circle at the right of $s_0+\epsilon$, shows that the $p_s$ have no roots larger than $b$. It follows that $\rho_1(p_s)\in (a,b)$ for $s$ close to $s_0$.
	
	If we
	write now $p_s(z)=(z-\rho_1(p_s)) q_s(z)$, 
	then $q_s$ has degree $n-1$ and continuous coefficients, so its roots are continuous by the induction hypothesis. But we have $\rho_i(p)=\rho_{i-1}(q)$ for $i\ge 2$.	
\end{proof}

\begin{remark}
	Even without the assumption that the roots are real, one can prove that there exist continuous functions $\rho_i(s):I\to \bbC$, $i=1, \dots, n$, such that the roots of $p_s$ are $\rho_1(s), \dots ,\rho_n(s)$ for all $s\in I$. The proof is more involved; see, for instance,~\cite[II.5.2]{K}.
	\end{remark}

We prove next two lemmas about polynomials with real coefficients and real roots.

\begin{lemma}\label{le:changing signs}
	Suppose the polynomial $p$ of degree $n$ has real coefficients, real roots, and the leading term positive.  Moreover, assume that there exist real numbers $a_{n+1}<a_n<\cdots<a_1$ such that $\rho_j(p)\in [a_{j+1}, a_j]$ for all $j=1, \dots, n$. Then $(-1)^{j-1}p(a_j)\ge 0$ for all $j=1, \dots, n$.
\end{lemma}

In other words, $p$ changes signs (not necessarily strictly) on each of the intervals $[a_{j+1}, a_j]$.

\begin{proof}
	We will use induction with respect to $n$. For $n=1$ the claim is obviously true. Suppose it is true up to $n-1$, and let $p$ be a polynomial of degree $n$ as in the statement of the lemma. There are two cases to consider.
	
	Suppose first that the roots of $p$ are exactly all points $a_j$ except some $a_{j_0}$. Then $p$ has only simple roots, so it changes signs in each of them. As  $p(x)>0$ for $x>a_1$, we have $p(x)<0$ on $(a_2, a_1)$, etc, up to $(-1)^{j_0-1}p(x)>0$ on $(a_{j_0+1}, a_{j_0-1})$. Therefore $(-1)^{j_0-1}p(a_{j_0})>0$; the other inequalities are trivial.
	
	In the remaining case, there is at least one root $\alpha$ of $p$ that is not among the points $a_j$; suppose  $\alpha\in (a_{j_0}, a_{j_0-1})$. If $p(z)=(z-\alpha)q(z)$, then $q$ has degree $n-1$ and satisfies the hypotheses of the lemma with respect to the points $a_j$ with $j\not=0$. Then $p(a_j)$ has the same sign as $q(a_j)$ for $j<j_0$ and opposite sign for $j>j_0$; from here it follows easily that the correct signs for $q$ (which we know true by the induction hypothesis) produce the correct signs for $p$.
\end{proof}

\begin{lemma}\label{le:properties of Phi in one variable}
	Suppose the polynomial $p$ has real coefficients and all roots real. Then
	\[
	(-1)^k \left(\frac{d}{dx}\right)^k \frac{p'}{p}(x) >0
	\]
	for all $k\in\bbN$ and $x>\rho_1(p)$.
	
	In particular, $\frac{p'}{p}$ is positive, nonincreasing, and convex for $x>\rho_1(p)$.
\end{lemma}

\begin{proof}
	If $p(z)=\prod_{i=1}^{n} (z-\rho_i(p))$, then $\frac{p'}{p}(z)= \sum_{i=1}^{n} \frac{1}{z-\rho_i(p)}$, and
	\[
	(-1)^k \left(\frac{d}{dx}\right)^k \frac{p'}{p}(x)= k! 
	\sum_{i=1}^{n} \frac{1}{(z-\rho_i(p))^{k+1}}.
	\]
	All terms in the last sum are positive for $x>\rho_1(p)$, so the lemma is proved.
\end{proof}

\subsection{Nice families}

Suppose $\FF=\{f_1,\dots f_m\}$ is a family of polynomials of the same degree~$n$. We denote 
\[
\rho^+_j(\FF):=	\max_{1\le i\le m}\rho_j(f_i), \quad\rho^-_j(\FF):=	\min_{1\le i\le m}\rho_j(f_i)
\]

\begin{definition}\label{de:nice family}
	For a family of polynomials $\FF=\{f_1,\dots f_m\}$ of the same degree $n $ a \emph{nice family} iff:
	\begin{enumerate}
		\item the coefficient of the dominant term of every $ f_j $ is positive;
		\item every $ f_j $ has all roots real;
		\item for all $ j=2,\dots, n $ we have
	\begin{equation}\label{eq:condition for nice families}
		\rho^+_j(\FF) \le\rho^-_{j-1}(\FF).
	\end{equation}
	
	\end{enumerate}
	
\end{definition}

The usual formulation (including~\cite{MSS}) is that the $f_i$s \emph{have a common interlacing}. Since the actual interlacing polynomial never enters our picture, we prefer this simpler phrasing.

\begin{lemma}\label{le:nice family:basic properties}

\begin{itemize}
	\item[(i)] $\{f_1,\dots f_m\}$ is nice iff every pair $\{f_r, f_s\}$, $ r\not= s $, is nice.
	
	\item[(ii)] Every subfamily of a nice family is nice.
	
	\item[(iii)] If $ a\in\bbR $, then  $\FF=\{f_1,\dots f_m\}$ is nice if and only if $\GG=\{(x-a)f_1,\dots (x-a)f_m\}$ is nice.
\end{itemize}
	
\end{lemma}

\begin{proof}
	(i) and (ii) are immediate. For (iii), there are several cases to consider:
	\begin{enumerate}
		\item If $a\in[\rho^-_{j_0}(\FF), \rho^+_{j_0}(\FF)]$ for some $j_0$, then
		\[
		\begin{split}
		\rho^\pm_j(\GG)&= \rho^\pm_j(\FF)\text{ for }j<j_0,\\
		\rho^+_{j_0}(\GG)&=\rho^+_{j_0}(\FF), \quad 	\rho^-_{j_0}(\GG)=	\rho^+_{j_0+1}(\GG)=1,\quad
		\rho^-_{j_0+1}(\GG)=\rho^-_{j_0}(\FF),\\
		\rho^\pm_j(\GG)&= \rho^\pm_{j-1}(\FF)\text{ for }j>j_0+1.
		\end{split}
		\]
		
		\item
		If $a\in(\rho^+_{j_0}(\FF), \rho^-_{j_0-1}(\FF))$ for some $j_0$, then
		\[
		\begin{split}
		\rho^\pm_j(\GG)&= \rho^\pm_j(\FF)\text{ for }j<j_0,\\
		\rho^\pm_{j_0}(\GG)&=a,\\
			\rho^\pm_j(\GG)&= \rho^\pm_{j-1}(\FF)\text{ for }j>j_0.
		\end{split}
		\]
	
	\end{enumerate} 
		The formulas in (1) are also valid if $a>\rho^+_1(\FF)$ (taking $j_0=1$) or $a<\rho^-_n(\FF)$ (taking $j_0=n+1$). In all these cases one can easily check that (iii) is true.
		\end{proof}	
	
	As a consequence of Lemma~\ref{le:nice family:basic properties}, in order to check that a family is nice we can always assume that it has no common zeros.

The main theorem of this section is the characterization of nice families that follows.

\begin{theorem}\label{th:nice families}
	Suppose $ f_{1}, \dots , f_{m} $ are all polynomials of degree $ n $, with positive dominant coefficients. The following are equivalent:
	\begin{enumerate}
		\item $\FF=\{f_1,\dots f_m\}$ is a nice family.
		
		\item Any convex combination of $ f_{1}, \dots , f_{m} $ has only real roots.
	\end{enumerate}
	
	If these conditions are satisfied, then for any $ j=1,\dots , n $ we have
	\begin{equation}\label{eq:roots of nice family}
	\min_{i}\rho_j(f_i) \le \rho_j(f)\le \max_{i}\rho_j(f_i)
	\end{equation}
	for any convex combination $f= \sum_k t_kf_k $.
\end{theorem}

\begin{proof}
(1)$\implies$(2). We may suppose by (ii) and (iii) of Lemma~\ref{le:nice family:basic properties}  that all coefficients $t_k$ are positive and that the family has no common zeros. In particular, if we denote $\rho^\pm_j=\rho^\pm_j(\FF)$, this implies $\rho^-_j<\rho^+_j\le\rho^-_{j-1} $ for all $j$.

We will apply Lemma~\ref{le:changing signs} to each of the polynomials $f_i$ and the points $\rho^-_n<\rho^-_{n-1}<\dots<\rho^-_1<\rho^+_1$. We obtain then, for each $i=1, \dots, m$, that $(-1)^j f_i(\rho^-_j)\ge0$ for all  $j$, and $ f_i(\rho^+_1)\ge0$. 

Fix $j$; since the family $\FF$ has no common zero, at least one of $f_i$ is nonzero in $\rho^-_j$, and so $(-1)^j f(\rho^-_j)>0$. Similarly, $ f(\rho^+_1)>0$. Therefore on each of the intervals $(\rho^-_j, \rho^-_{j-1})$, as well as on $(\rho^-_1, \rho^+_1)$, $f$ changes sign (strictly), and therefore must have a root in the interior. 
Since there are $n$ intervals, we have thus found $n$ roots of $f$, and so all its roots are real. Moreover, we have obtained $\rho_j(f)>\rho_j^-$ for all $j$. 

On the other hand, we might have used, in applying Lemma~\ref{le:changing signs} to  the polynomials $f_i$,  the points $\rho^-_n<\rho^+_{n}< \rho^+_{n-1}\dots<\rho^+_1$ instead of $\rho^-_n<\rho^-_{n-1}<\dots<\rho^-_1<\rho^+_1$. A similar argument yields then $\rho_j(f)<\rho_j^+$ for all $j$. Therefore 
the inequalities~\eqref{eq:roots of nice family} are proved.

(2)$ \implies $(1). According to Lemma~\ref{le:nice family:basic properties} it is enough to prove the implication for two functions $f_1, f_2$, and we may also suppose that they have no common roots. Fix $2\le j\le n$; we have to prove that $\rho_j^+\le \rho_{j-1}^-$.
Denote $ f_t=tf_1+(1-t)f_2 $ ($ 0\le t\le 1 $). By Corollary~\ref{co:continuity of real roots} the function $t\mapsto\rho_j(f_t)$ is continuous on $ [0,1] $ and takes only real values; so its values for $0<t<1$ cover the interval $(\rho_j^-, \rho_j^+)$. It follows that this interval cannot contain a root of either $f_1$ or $f_2$, since a common root of, say, $f_1$ and $f_t$ is also a root of $f_2$.

Suppose then first that $f_1$ and $f_2$ have only simple roots. Then the intervals $[\rho_j^-, \rho_j^+]$ and $[\rho_{j-1}^-, \rho_{j-1}^+]$ have all four endpoints disjoint, and by definition $\rho_j^-< \rho_{j-1}^-$. If $\rho_{j-1}^- \in (\rho_j^-, \rho_j^+)$, this would contradict the conclusion of the preceding paragraph. So $\rho_{j-1}^->\rho_j^+$ and~\eqref{eq:condition for nice families} is proved.

To obtain the general case, note first that $f_t$ has all roots simple for $0<t<1$. Indeed, a multiple solution $x$ of $f_t=0$ would also be a multiple solution of $\frac{f_2}{f_1}=\frac{t}{t-1}$. But it is easy to see (draw the graph!) that then  $\frac{f_2}{f_1}=\frac{t'}{t'-1}$ has a single root in some interval $(x-\epsilon,x+\epsilon)$ for at least some $t'$ close to $t$ (slightly larger or slightly smaller). However, from Theorem~\ref{th:Cauchy} it follows that $f_{t'}$ has more than one root in the disc $|z-x|<\epsilon$, and so $f_{t'}$ would not have all roots real.

To end the proof, we apply the first step to $f_\epsilon$ and $f_{1-\epsilon}$ ($\epsilon>0$), which have only simple roots. Then we let $\epsilon\to 0$ and use Corollary~\ref{co:continuity of real roots} to obtain inequality~\eqref{eq:condition for nice families}.
\end{proof}

\section{Several variables: real stable polynomials}\label{se:real stable}

\subsection{General facts}

Denote $\bbH=\{ z\in \bbC: \Im z>0 \}$.

\begin{definition}
	A polynomial $p(z_1,\dots, z_m)$ is called \emph{real stable} it has real coefficients and it has no zeros in $\bbH^m$.
\end{definition}

In case $m=1$ a real stable polynomial is a polynomial that has real coefficients and real zeros. Genuine examples in several variables are produced by the next lemma.

\begin{lemma}\label{le:example of real stable} 
	If
	$A_1,\dots, A_m\in M_d(\bbC)$ are  positive matrices, then the polynomial
	\begin{equation}\label{eq:definition of q}
	q(z, z_1,\dots, z_m)=\det(zI_d + \sum_{i=1}^{m}z_i A_i)
	\end{equation}
	is real stable.
\end{lemma}

\begin{proof}
	It is immediate from the definition that 
	\[
	q(\bar z, \bar z_1, \dots, \bar z_m)=\overline{q(z, z_1,\dots, z_m)},
	\]
	whence the coefficients of $q$ are real.
	
	Assume that $ q(z, z_1,\dots, z_m)=0 $, and $\Im z, \Im z_i>0$. Since $ zI_d + \sum_{i=1}^{m}z_i A_i $ is not invertible, there exists $\xi\in\bbC^d$, $\xi\not=0$, such that
	\[
	0=\< (zI_d + \sum_{i=1}^{m}z_i A_i)\xi, \xi\>=z\|\xi\|^2 +\sum_{i=1}^{m} z_i\<A_i\xi, \xi\>, 
	\]
	and so
	\[
	0=\Im z\|\xi\|^2 +\sum_{i=1}^{m}\Im z_i\<A_i\xi, \xi\> .
	\]
	This is a contradiction, since $\Im z\|\xi\|^2 >0$ and $\Im z_i\<A_i\xi, \xi\>\ge 0$ for all $i$.
\end{proof}

The next theorem gives the basic properties of real stable polynomials. Denote, for simplicity, by $\partial_i$ the partial derivative $ \frac{\partial}{\partial z_i} $. 

\begin{theorem}\label{th:basic properties of stable}
Suppose $p$ is a real stable polynomial.

\begin{itemize}
	\item[(i)] If $m>1$ and $t\in\bbR$, then $p(z_1,\dots, z_{m-1}, t)$ is either real stable or identically zero.
	
	\item[ (ii)] If $t\in \bbR$, then $(1+t\partial_m)p$ is real stable.
\end{itemize}
\end{theorem}

\begin{proof}
	(i) Obviously $p(z_1,\dots, z_{m-1}, t)$ has real coefficients. Suppose it is not identically zero. If $\Im w>0$ is fixed, then the polynomial $p(z_1,\dots, z_{m-1}, w)$ is real stable by definition. Therefore all polynomials $p(z_1,\dots, z_{m-1}, t+\frac{i}{n})$ (for $n\in\bbN$) are real stable. We let then $n\to\infty$ and apply Corollary~\ref{co:Herglotz and other's} to $D=\bbH^m$ to obtain the desired result.
	
	(ii) We may assume $t\not=0$ (otherwise there is nothing to prove). Suppose $(1+t\partial_m)p(z_1,\dots, z_m)=0$ for some $(z_1, \dots, z_m)\in \bbH^m$. Since $p$ is real stable, $p(z_1,\dots, z_m)\not=0$. The one-variable polynomial $q(z):=p(z_1,\dots,z_{m-1}, z)$ has no roots with positive imaginary part (in particular, $q(z_m)\not=0$), so we may write
	\[
	q(z)=c\prod_{i=1}^{n}(z-w_i), \qquad \Im w_i\le 0.
	\]
	Therefore
	\[
	0= (1+t\partial_m)p(z_1,\dots, z_m) = (q+tq')(z_m)=
	q(z_m)\left(1+t\frac{q'(z_m)}{q(z_m)} \right),
	\]
	and, since $q(z_m)\not=0$,
	\[
	0=1+t\sum_{i=1}^{n} \frac{1}{z_m-w_i}=1+t\sum_{i=1}^{n} \frac{\overline{z_m-w_i}}{|z_m-w_i|^2}.
	\]
	Taking the imaginary part, we obtain
	\[
	t \sum_{i=1}^{n} \frac{\Im w_i-\Im z_m}{|z_m-w_i|^2}=0
	\]
	which is a contradiction, since $t\not=0$ and $\Im w_i-\Im z_m<0$ for all $i$.
\end{proof}

We will also need a lemma that uses a standard result in algebraic geometry, namely B\'ezout's Theorem (which can be found in any standard text).

\begin{lemma}\label{le:bezout and things}
	Suppose $p(z,w)$ is a nonconstant polynomial in two variables, of degree $n$ in $w$, which is irreducible over $\bbR$.
		 There is a finite set $F\in\bbC$ such that, if $p(z_0,w_0)=0$ and $z_0\not\in F$, then:
		 \begin{enumerate}
		 	\item the equation $p(z_0,w)=0$ has $n$ distinct solutions;
		 	\item for each of these solutions $(z_0,w_0)$ we have
$\frac{\partial p}{\partial w}(z_0, w)\not=0$.
		 \end{enumerate} 
	
\end{lemma}

\begin{proof} First, if $p(z,w)=q(z)w^n+\dots$, then the roots of $q$ form a finite set $F_1$.
	
	Secondly, 	if $p$ is irreducible, then $p$ and $\frac{\partial p}{\partial w}$ are coprime over $\bbR$, and hence also over $\bbC$.  B\'ezout's Theorem in algebraic geometry states that  two curves defined by coprime equations have only a finite number of common points, so this is true about the sets defined by $p(z,w)=0$ and $\frac{\partial p}{\partial w}(z,w)=0$. Let $F_2$ be the set of the projections of these points onto the first coordinate. The set $F=F_1\cup F_2$ has the required properties.
\end{proof}

\subsection{The barrier function}
Our eventual purpose in this subsection is to obtain  estimates on the roots of real stable polynomials; more precisely, we want to show that a restriction on the roots of a real stable polynomial $p$ may imply a restriction on the roots of $(1-\partial_i)p$ (which is also real stable by Theorem~\ref{th:basic properties of stable}). 

We will often use the restriction of a polynomial in $m$ complex variables to $\bbR^m\subset \bbC^m$. To make things easier to follow, we will be consistent in this subsection with the following notation: $z,w$ will belong to $\bbC^m$ (and corresponding subscripted letters in $\bbC$), while $x,y,s,t$ will be in $\bbR^m$ (and corresponding subscripted letters in $\bbR$). If $x=(x_1, \dots, x_m)\in \bbR^m$, then $\{ y\ge  x \}$ will denote $\{ y=(y_1, \dots, y_m)\in \bbR^m : y_i\ge x_i \text{ for all }i=1, \dots, m \}$. 

The main tool is a certain function associated to $p$ called  the \emph{barrier function}, whose one-dimensional version has already been met in Lemma~\ref{le:properties of Phi in one variable}. It is defined wherever $p\not=0$ by 
 $\Phi_p^i=\frac{\partial_i p}{p}$; if $p(x)>0$ it can also be written as $\Phi_p^i(x)=\partial_i(\log p)(x)$. The argument of the barrier function will always actually be in $\bbR^m$.

The connection of the barrier function with our problem is given by the simple observation that if $p(x)\not=0$ and $(1-\partial_i)p(x)=0$, then $\Phi^i_p(x)=1$. So, in particular, a set on which  $0\le\Phi^i_p<1$ does not contain zeros of  $(1-\partial_i)p$. To determine such sets, the basic result is the next lemma, which is a multidimensional extension of Lemma~\ref{le:properties of Phi in one variable}.

\begin{lemma}\label{le:basic conditions on Phi}
	Suppose $x\in\bbR^m$, and $p(z_1,\dots, z_m)$ is a real stable polynomial that has no roots in $\{y\ge x \}$, then
	\[
	(-1)^k \frac{\partial^k}{\partial z^k_j} \Phi^i_p(x')\ge 0
	\]
	for any $k\ge 0$, $1\le i,j\le m$, and $x'\ge x$.
	
	In particular, if $e_j$ is one of the canonical basis vectors in $\bbC^m$, then $t\mapsto \Phi^i_p(x+t e_j)$ is positive, nonincreasing and convex on $[0,\infty]$.
\end{lemma}

\begin{proof}
	 The assertion reduces to Lemma~\ref{le:properties of Phi in one variable} for $m=1$ or for $k=0$; and also for $i=j$, since then  fixing all  variables except the $i$th reduces the problem to the one variable case. 
	 
	 In the general case,  it is enough to do it for $y=x$, since if $p$ has no roots in $\{y\ge x \}$, then it has no roots in $\{y\ge x'\}$ for all $x'\ge x$. By fixing all variables except $i$ and $j$, we may assume that $m=2$, $i=1$, $j=2$, $k\ge 1$. Moreover,  we may also assume that $p>0$ on $\{y\ge x \}$ (otherwise we work with $-p$, since $\Phi^i_p=\Phi^i_{-p}$).

	 So we have to prove that, if $p(z_1, z_2)$ is a real stable polynomial which has no zeros in $\{y_1\ge x_1, \ y_2\ge x_2 \}$, then
	 \[
	 0\le (-1)^k \frac{\partial^k}{\partial z^k_2} \Phi^1_p(x_1, x_2)
	 = (-1)^k \frac{\partial^k}{\partial z^k_2} \left(\frac{\partial}{\partial z_1}\log p \right) (x_1, x_2)
	 =\frac{\partial}{\partial z_1} \left((-1)^k \frac{\partial^k}{\partial z^k_2}\log p   \right) (x_1, x_2).
	 \]
	 
	 We will in fact prove that the map
	 \[
	 t\mapsto (-1)^k \frac{\partial^k}{\partial z^k_2}\log p (t, x_2)
	 \]
	 is increasing for $t\ge x_1$. It is enough to achieve this for $p$ irreducible over $\bbR$, since, if $p=p_1 p_2$ is real stable and has no roots in $\{y\ge x \}$, then the same is true for $p_1$ and $p_2$, and obviously
	 \[
	 (-1)^k \frac{\partial^k}{\partial z^k_2}\log p (t, x_2)=(-1)^k \frac{\partial^k}{\partial z^k_2}\log p_1 (t, x_2)+(-1)^k \frac{\partial^k}{\partial z^k_2}\log p_2 (t, x_2).
	 \]

	 Suppose then that $p$ is irreducible. 
	 For $t\ge x_1$ fixed, the polynomial $p(t,z)$ is real stable, and thus has all roots real; denote them, as in Section~\ref{se:univ poly}, by $\rho_1(t)\ge \dots\ge \rho_n(t)$. 
	 
	 Applying to $p$ Lemma~\ref{le:bezout and things}, take $t\ge x_1$  that does not belong to the finite set $F$ therein. The functions $\rho_i(t)$ are therefore differentiable in $t$, and we have 
	 \begin{equation}\label{eq:formula for p with roots}
	 p(t, z)=c(t)\prod_{i=1}^{n}(z-\rho_i(t))
	 \end{equation}
	  Therefore
	 \begin{equation}\label{eq:the function that has to be increasing}
	 \begin{split}
	  \left((-1)^k \frac{\partial^k}{\partial z^k_2}\log p   \right) ((-1)^k \frac{\partial^k}{\partial z^k_2}\log p) (t, x_2)&=
	  (-1)^k \frac{\partial^k}{\partial z^k_2}
	  \left(\sum_{i=1}^{n} \log(z-\rho_i(t))\right)\Big|_{z=x_2} \\ 
	  &=- \sum_{i=1}^{n} \frac{(k-1)!}{(x_2-\rho_i(t))^k}.
	 \end{split}
	 \end{equation}
	  If $t\ge x_1$, we cannot have $\rho_i(t)\ge x_2$ since then $(t, \rho_i(t))$ would be a root of $p$ in $\{y\ge x \}$, contrary to the assumption. Thus $x_2-\rho_i(t)>0$, and in order to show that the function in~\eqref{eq:the function that has to be increasing} is increasing, it is enough to show that $t\mapsto \rho_i(t)$ is decreasing for $t\ge x_1$ and all $i$.
	  
Now all $\rho_i$s are differentiable for $t\ge x_1$, $t\notin F$. To show that they are decreasing, it is enough to show that $\rho_i'(t)\le 0$ for such $t$. Suppose then that there exists $i\in\{1,\dots, n \}$ and $t\ge x_1$ such that $\rho_i'(t)>0$; let $s=\rho_i(t)$. 
	 Since $\frac{\partial p}{\partial z_2} (t, s)\not=0$, we may apply the (complex) implicit function theorem in a neighborhood of $(t,s)$ (in $\bbC^2$). We obtain that the solutions of $p(z_1, z_2)=0$ therein are of the form $(z_1,g(z_1))$ for some locally defined analytic function of one variable $g$, which by analytic continuation has to be an extension of $\rho$ to a complex neighborhood of $t$. So $g'(t)=\rho'_i(t)$, and in the neighborhood of $t$ we have
	 \[
	 g(z_1)= t+\rho'_i(t) (z_1-t)+O(|z_1-t|^2).
	 \]
	 If $\Im z_1>0$ and small, one  also has $\Im g(z_1)>0$. We obtain thus the zero $(z_1, g(z_1))$ of $p$ in $\bbH^2$, contradicting the real stability of $p$. This ends the proof of the lemma.
\end{proof}

\begin{corollary}\label{co:Phi(y) le Phi(x)}
	Suppose $x\in\bbR^m$, and $p$ is a real stable polynomial, without zeros in $ \{y\ge x\} $. Then $\Phi^j_p(y)\le \Phi^j_p(x)$ for any $y\ge x$ and $j=1, \dots, m$.
\end{corollary}

\begin{proof}
If $p$ has no zeros in  $ \{y\ge x\} $, obviously it has no zeros in  $ \{y\ge x'\} $ for any $x'\ge x$.  Therefore,
	by Lemma~\ref{le:basic conditions on Phi}, the function $t\mapsto \Phi^j_p(x'+te_i)$ is nonincreasing on $[0,\infty)$ for any $i=1, \dots, m$. We have then
	\[
	\Phi^j_p(x_1,\dots, x_m)\ge
		\Phi^j_p(y_1,x_2,\dots, x_m)\ge
		\Phi^j_p(y_1,y_2,x_3,\dots, x_m)\ge
		\dots
		\ge 	\Phi^j_p(y_1,\dots, y_m)\qedhere
	\] 
	\end{proof}

The main monotonicity and convexity properties of $\Phi^i_p$ are put to work in the next lemma to obtain a restriction on the location of zeros of $(1-\partial_j)p$. As noted above, we will use the condition $\Phi^j_p<1$, but in a more precise variant which will lends itself to iteration.

\begin{lemma}\label{le:applying 1-partial to p]}
	Let $x\in\bbR^m$, and $p$  a real stable polynomial, without zeros in $ \{y\ge x\} $. Suppose also that 
	\[
	\Phi^j_p(x)+\frac{1}{\delta}\le 1
	\]
	for some $j\in\{1, \dots, m\}$ and $\delta>0$.
	
	Then:
	\begin{itemize}
		\item[(i)] $(1-\partial_j)p$ has no zeros in  $ \{y\ge x\} $.
		
		\item[(ii)] For any $i=1, \dots, m$ we have
		\[
		\Phi^i_{(1-\partial_j)p}(x+\delta e_j) \le \Phi^i_p(x).
		\]
	\end{itemize}
	
\end{lemma}

\begin{proof}
	By Corollary~\ref{co:Phi(y) le Phi(x)} we have 
	\[
	\frac{\partial_j p(y)}{p(y)}=\Phi(y)\le \Phi(x)\le 1-\frac{1}{\delta}<1,
	\]
	so $\partial_j p(y)\not= p(y)$, or $(1-\partial_j)p(y)\not=0$.
	
	To prove (ii), note first that $(1-\partial_j)p=p(1-\Phi^j_p)$, whence 
	$\log [(1-\partial_j)p]=\log p+ \log(1-\Phi^j_p) $, so, by differentiating,
	\[
	\Phi^i_{(1-\partial_j)p}=\Phi^i_p-\frac{\partial_i \Phi^j_p}{1-\Phi^j_p}.
	\]
	The required inequality becomes then
	\begin{equation}\label{eq:required Phi}
	-\frac{\partial_i\Phi^j_p(x+\delta e_j)}{1-\Phi^j_p(x+\delta e_j)}\le \Phi^i_p(x)- \Phi^i_p(x+\delta e_j).
	\end{equation}
	By Corollary~\ref{co:Phi(y) le Phi(x)} we have
	\[
	\Phi^j_p(x+\delta e_j)\le \Phi^j_p(x)\le 1-\frac{1}{\delta},
	\]
	or
	\[
	\frac{1}{1-\Phi^j_p(x+\delta e_j)}\le \delta.
	\]
	
	Further on, $p$ has no zeros in $\{ y\ge x+\delta e_j \}$, so Lemma~\ref{le:basic conditions on Phi} (applied in $x+\delta e_j$) implies, in particular, that $-\partial_i\Phi^j_p(x+\delta e_j)\ge 0 $, whence
	\[
	-\frac{\partial_i\Phi^j_p(x+\delta e_j)}{1-\Phi^j_p(x+\delta e_j)}\le- \delta \partial_i\Phi^j_p(x+\delta e_j).
	\]
	To prove~\eqref{eq:required Phi}, it is then enough to show that
	\[
	- \delta \partial_i\Phi^j_p(x+\delta e_j) \le \Phi^i_p(x)- \Phi^i_p(x+\delta e_j).
	\]
	Using $\partial_i\Phi^j_p(x+\delta e_j)=\partial_j\Phi^i_p(x+\delta e_j)$, the inequality can be written
	\[
	\Phi^i_p(x+\delta e_j)\le \Phi^i_p(x)+ \delta \partial_j\Phi^i_p(x+\delta e_j).
	\]
	This, however, is an immediate consequence of the convexity of the function $t\mapsto\Phi^i_p(x+t e_j) $, that has been proved in Lemma~\ref{le:basic conditions on Phi}. 
\end{proof}

Finally,  the next theorem is the main result of this section that we will use in the sequel.

\begin{theorem}\label{th:real stable no zeros after 1-partial}
	Let $x\in\bbR^m$, and $p$  a real stable polynomial, without zeros in $ \{y\ge x\} $. Suppose also that 
\[
\Phi^j_p(x_1,\dots, x_m)+\frac{1}{\delta}\le 1
\]
for some $\delta>0$ and $ j=1,\dots, m $. Then
\[
\prod_{i=1}^{m} (1-\partial_i)p
\]
has no zeros in $ \{y\ge x+\tilde{\delta}\} $, where $ \tilde{\delta}:=(\delta,\dots, \delta)\in \bbR^m $.
\end{theorem}

\begin{proof}
	 The proof follows by applying  Lemma~\ref{le:applying 1-partial to p]} successively to $j=1$ and $x$, then to $j=2$ and $x+\delta e_1$, etc. 
\end{proof}

\section{Characteristic and mixed characteristic polynomials}\label{se:mixed char}

\subsection{Mixed characteristic polynomial}

We intend now to apply the results of Section~\ref{se:real stable} to polynomials related to matrices. Our final goal is to estimates eigenvalues; that is, roots of the characteristic polynomial. But we will first consider another polynomial, attached to a tuple of matrices. 

\begin{definition}\label{de:mixed char}
	If $A_1,\dots, A_m\in M_d(\bbC)$, then the \emph{mixed characteristic polynomial} of the matrices $A_i$ is defined by the formula
\begin{equation}\label{eq:def of mu}
\mu[A_1,\dots, A_m](z )=
\prod_{i=1}^{m} (1-\partial_i) \det(z  I_d+\sum_{i=1}^{m}z_i A_i) \Big|_{z_1=\dots=z_m=0}.
\end{equation}
\end{definition}

It is easily seen that if we fix $m-1$ of the matrices $A_1, \dots, A_m$, then  $ \mu[A_1,\dots, A_m](z ) $ is of degree 1 in the entries of the remaining matrix. Indeed, if we develop the determinant that enters~\eqref{eq:def of mu}, then any term that contains a product of, say, $k$ entries of $A_j$ has also the factor $z_j^{k}$. If we apply $(1-\partial_j)$, we are left with  $z_j^{k-1}$, and if $k\ge 2$ this terms becomes 0 if $z_j=0$.

\begin{example}
	For one or two matrices we have
	\[
	\begin{split}
\mu[A_1](z)&=z^d-z^{d-1}\tr A_1\quad\text{ if }m=1,\\
	\mu[A_1, A_2](z)&= z^d-z^{d-1}(\tr A_1+\tr A_2) +z^{d-2} (\tr A_1 \tr A_2- \tr(A_1A_2)) \quad\text{ if }m=2.
	\end{split}
	\]

In the general case, the coefficients of $ \mu[A_1,\dots, A_m](z ) $ are certain expressions in the traces of monomials in $A_1, \dots, A_m$ that are well known in the invariant theory of matrices (see~\cite{P}). 
\end{example}

The results in Section~\ref{se:real stable} have consequences for the mixed characteristic polynomials.

\begin{theorem}\label{th:mixed has real roots}
	Suppose $A_1,\dots, A_m\in M_d(\bbC)$ are positive matrices. Then $\mu[A_1, \dots, A_m](z )$ has only real roots.
\end{theorem}

\begin{proof}
	We have seen in Lemma~\ref{le:example of real stable} that the polynomial $q$ defined by~\eqref{eq:definition of q} is real stable. But $\mu[A_1,\dots, A_m]$ is obtained from $q$ by first applying $(1-\partial_i)$ for $ i=1,\dots, m $ and then specializing to $ z_1=\dots=z_m=0 $. By Theorem~\ref{th:basic properties of stable}, these operations preserve the real stable character. So  $\mu[A_1,\dots, A_m]$ is a real stable polynomial of one variable, which means exactly that it has real roots.
\end{proof}

Remember   Jacobi's formula for the derivative of the determinant of an invertible matrix:
\begin{equation}\label{eq:jacobi}
\frac{(\det M(t))'}{\det M(t)}=\tr\left(M(t)^{-1}M'(t)\right).
\end{equation}

\begin{theorem}\label{th:sum=1, trace small}
	Suppose $A_1,\dots, A_m\in M_d(\bbC)$ are positive matrices, such that $ \sum_{i=1}^{m}A_i=I_d $ and $\tr A_i\le \epsilon$ for each $i=1,\dots, m$. Then any root of $ \mu[A_1,\dots, A_m] $ is smaller than $ (1+\sqrt{\epsilon})^2 $.
\end{theorem}

\begin{proof}
	The polynomial 
	\[
	p(z):=\det(\sum_{i=1}^{m}z_iA_i).
	\]
	is real stable, being the specialization of the polynomial $q$ in~\eqref{eq:definition of q} to $z=0$. If $t>0$
	and $\tilde{t}:=(t,\dots, t)\in \bbC^d$, then, for $y\ge \tilde{t}$ we have $ \sum_{i=1}^{m}y_iA_i\ge\sum_{i=1}^{m}tA_i=tI_d  $. Therefore $ \sum_{i=1}^{m}y_iA_i $ is invertible, and  $p(y)\not=0$.
	
	We may apply Jacobi's formula~\eqref{eq:jacobi} in order to compute the barrier function $\Phi_p^j$, and we obtain
	\[
	\Phi_p^j(x_1,\dots, x_m)=\tr ((\sum_{i=1}^{m}z_iA_i)^{-1}A_j).
	\]
	In particular, if $t>0$, then
	\[
	\Phi_p^j(t,\dots, t)=\tr(t^{-1}A_j)\le \frac{\epsilon}{t}.
	\]
It follows then from Theorem~\ref{th:real stable no zeros after 1-partial} that, if we  $t,\delta>0$ are such that $ \frac{\epsilon}{t}+\frac{1}{\delta}\le 1 $, then $ \prod_{i=1}^{m}(1-\partial_i)p $ has no zeros in $\{y\ge (t+\delta, \dots, t+\delta)\}	$. The choice $t=\epsilon+\sqrt{\epsilon}$, $ \delta=1+\sqrt{\epsilon} $ (which can easily be shown to be optimal) yields $ t+\delta=(1+\sqrt{\epsilon})^2 $, and therefore $p$ has no roots $y$ with $ y_i \ge (1+\sqrt{\epsilon})^2 \} $ for all $i$.

Now, using the relation $\sum_{i=1}^{m}A_i=1$, one obtains
\[
\begin{split}
\mu[A_1, \dots, A_m](z)&= \prod_{i=1}^{m} (1-\partial_i)\det(zI_d+\sum_{i=1}^{m}z_iA_i )\big|_{z_1=\dots=z_m=0}\\
&= \prod_{i=1}^{m} (1-\partial_i)\det(\sum_{i=1}^{m}w_iA_i)\big|_{w_1=\dots=w_m=z}\\
&= \prod_{i=1}^{m} (1-\partial_i) p(z,z, \dots, z),
\end{split}
\]
	which cannot be zero if $z\ge (1+\sqrt{\epsilon})^2$. Therefore all roots of $\mu$ are smaller than $ (1+\sqrt{\epsilon})^2$.
\end{proof}

\subsection{Decomposing in rank one matrices and the characteristic polynomial}

In an important particular case  the mixed characteristic polynomial coincides with a usual characteristic polynomial. Remember this  is defined, for $ A\in M_d(\bbC) $,  by 
$p_A(z )=\det(z  I_d-A)$.

\begin{lemma}\label{le:determinant multilinear}
	Suppose $B, A_1,\dots, A_m\in M_d(\bbC)$, and $A_1, \dots, A_m$ have rank one. Then the polynomial
	\[
	(z_1,\dots, z_m)\mapsto \det(B+z_1A_1+\cdots+ z_mA_m)
	\]
	is of degree $\le 1$ separately in each variable.
\end{lemma}

\begin{proof}
	By fixing all the variables except one, we have to show that, for any $ B, A_1\in M_d(\bbC) $, $A_1$ of rank one, the function
	\[
	z\mapsto \det(B+zA_1)
	\]
	is of degree at most 1. This is obvious if we choose a basis in which the first vector spans the image of $ A_1 $, and we develop the determinant with respect to the first row.
\end{proof}

Suppose now $ p(z_1, \dots, z_m) $ is  a polynomial  of degree $\le 1$ separately in each variable. Then $p$ is equal to its Taylor expansion at the origin of order 1 in each variable, that is:
\[
p(z_1, \dots, z_m)=\sum_{\epsilon_i\in\{0,1\}} c_{\epsilon_1,\dots,\epsilon_m} z_1^{\epsilon_1}\cdots z_m^{\epsilon_m},
\]
with 
\[
c_{\epsilon_1,\dots,\epsilon_m}=\partial_1^{\epsilon_1}\cdots \partial_m^{\epsilon_m}p(w_1,\dots, w_m)\big|_{w_1=\dots=w_m=0}.
\]
Therefore
\[
\begin{split}
p(z_1, \dots, z_m)&= \sum_{\epsilon_i\in\{0,1\}} z_1^{\epsilon_1}\cdots z_m^{\epsilon_m}\partial_1^{\epsilon_1}\cdots \partial_m^{\epsilon_m}p(w_1,\dots, w_m)\big|_{w_1=\dots=w_m=0}\\
&= \prod_{i=1}^{m} (1+z_i\partial_i) p(w_1,\dots, w_m)\big|_{w_1=\dots=w_m=0}. 
\end{split}
\]
In the case of the polynomial in Lemma~\ref{le:determinant multilinear}, this formula becomes
\[
\det(B+\sum_{i=1}^{m}z_i A_i)=
\prod_{i=1}^{m} (1+z_i\partial_i) \det(B+\sum_{i=1}^{m}w_i A_i)\big|_{w_1=\dots=w_m=0}.
\]

In fact, we are interested by this last formula precisely when
 $B=z  I_d$ and all $z_i=-1$. We obtain then the next theorem.

\begin{theorem}\label{th:charact=mixed char}
Suppose $ A_1,\dots, A_m\in M_d(\bbC)$ have rank one. If $ A=A_1+\cdots+A_m $, then
\[
p_A(z )=\mu[A_1,\dots, A_m](z ).
\]	
\end{theorem}

\begin{remark}\label{re:operators and matrices}
	The mixed characteristic polynomial and the usual characteristic polynomial are invariant with respect to a change of basis. So, although we have spoken about matrices for convenience, the statements of Theorems~\ref{th:sum=1, trace small} and~\ref{th:charact=mixed char} can be stated for $A_1, \dots, A_m\in \LL(V)$, where $\LL(V)$ denotes the space of linear operators on the finite dimensional vector space~$V$.
\end{remark}

\section{Randomisation}\label{se:random}

\subsection{Random matrices and determinants}

Let $ (\Omega, \mathbf{p}) $ be a finite probability space. If $X$ is a random variable on $\Omega$, the \emph{expectation} (or \emph{average}) of $\bbE (X)$ is defined, as usually, by
\[
\bbE (X):=\sum_{\omega\in\Omega}\mathbf{p}(\omega) X(\omega).
\]
If $X_1, \dots, X_m$ are \emph{independent} random variables, then, in particular, we have
\begin{equation}\label{eq:independent expectations}
\bbE(X_1\cdots X_m)=\bbE(X_1)\cdots \bbE(X_m).
\end{equation}

We will  use random matrices $\mathbf{A}(\omega)\in M_d(\bbC)$, whose entries are random variables; then $ \bbE(\mathbf{A}) $ is the matrix whose entries are the expectations of the corresponding entries of $\mathbf{A}$. The random matrices $\mathbf{A}_1, \mathbf{A}_2$ are called independent if any entry of $\mathbf{A}_1$ is independent  of every entry of $\mathbf{A}_2$. Also, when we say that a random matrix $\mathbf{A}(\omega)$ has rank one, this means that $\mathbf{A}(\omega)$ has rank one for any $\omega\in\Omega$.

The characteristic polynomial $p_\mathbf{A}$ of a random matrix $\mathbf{A}$ is also a random variable, by which we mean that its coeficients are random variables. Then the polynomial $\bbE(p_\mathbf{A})$ has as coefficients the expectations of the coefficients of $ \mathbf{A} $.

\begin{theorem}\label{th:expectation of char pol}
	Suppose $ \mathbf{A}_1(\omega), \dots, \mathbf{A}_m(\omega) $ are independent rank one random matrices in $M_d(\bbC)$, and $\mathbf{A}=\mathbf{A}_1+\dots+ \mathbf{A}_m$. Then
	\[
	\bbE(p_\mathbf{A})=\mu[\bbE(\mathbf{A}_1), \dots, \bbE(\mathbf{A}_m)].
	\]
\end{theorem}

\begin{proof}
	By Theorem~\ref{th:charact=mixed char} we have, for each $ \omega\in\Omega $, $ p_{\mathbf{A}(\omega)} = \mu[\mathbf{A}_1(\omega), \dots, \mathbf{A}_m(\omega)] $. By taking expectations, 
	\[
	\bbE(p_\mathbf{A})=\bbE (\mu[\mathbf{A}_1(\omega), \dots, \mathbf{A}_m(\omega)]).
	\]
	Now independence of $\mathbf{A}_i$s combined with the multilinearity of $\mu$ implies that
	\[
	\bbE (\mu[\mathbf{A}_1(\omega), \dots, \mathbf{A}_m(\omega)])=\mu[\bbE(\mathbf{A}_1), \dots, \bbE(\mathbf{A}_m)],
	\]
	which ends the proof.
\end{proof}

We can say more if we also assume that the $ \mathbf{A}_i $s are all positive.

\begin{theorem}\label{th:min root le mixed char of expectation}
Suppose $ \mathbf{A}_1(\omega), \dots, \mathbf{A}_m(\omega) $ are independent rank one positive random matrices in $M_d(\bbC)$, and $\mathbf{A}=\mathbf{A}_1+\dots+ \mathbf{A}_m$.	Then, for  any $ j=1,\dots, d $, we have
\[
\min_{\omega\in\Omega} \rho_j(p_{\mathbf{A}(\omega)}) 
\le \rho_j(\mu[\bbE(\mathbf{A}_1), \dots, \bbE(\mathbf{A}_m)])
\le \max_{\omega\in\Omega} \rho_j(p_{\mathbf{A}(\omega)}).
\]
\end{theorem}

\begin{proof}
	We prove only the left hand side inequality; the right hand side is similar. It is enough to show that for any $ i=1, \dots  $ we have
	\begin{equation}\label{eq:going from i-1 to i}
	\begin{split}
	&	\min_{\omega\in\Omega}  \rho_j(\mu[\mathbf{A}_1(\omega), \dots,\mathbf{A}_{i-1}(\omega),
		\mathbf{A}_i(\omega), \bbE(\mathbf{A}_{i+1}), \dots, \bbE(\mathbf{A}_m)])\\
	&\qquad\qquad		\le
		\min_{\omega\in\Omega}  \rho_j(\mu[\mathbf{A}_1(\omega), \dots,\mathbf{A}_{i-1}(\omega), \bbE(\mathbf{A}_i), \bbE(\mathbf{A}_{i+1}), \dots, \bbE(\mathbf{A}_m)]).
	\end{split}
	\end{equation}
	Indeed, for $ i=m $ the left hand side coincides with $ \min_{\omega\in\Omega} \rho_j(p_{\mathbf{A}(\omega)})  $ by Theorem~\ref{th:charact=mixed char}, while for $i=1$ the right hand side is precisely $ \rho_j(\mu[\bbE(\mathbf{A}_1), \dots, \bbE(\mathbf{A}_m)]) $. The chain of inequalities corresponding to $i=1,2,\dots, m$ proves then the theorem.
	
	Fix then $i$ and $\omega\in\Omega$, and consider the family of polynomials
	\[
	f_{\omega'}=\mu[\mathbf{A}_1(\omega), \dots, \mathbf{A}_{i-1}(\omega), \mathbf{A}_i(\omega'), \bbE(\mathbf{A}_{i+1}), \dots, \bbE(\mathbf{A}_m)], \quad \omega'\in\Omega.
	\]
	Take $c_{\omega'}\ge0$, with $\sum_{\omega'\in\Omega}c_{\omega'}=1$. By the multilinearity of the mixed characteristic polynomial, we have
	\[
	\begin{split}
	\sum_{\omega'\in\Omega}c_{\omega'} 	f_{\omega'}&
	= \sum_{\omega'\in\Omega}c_{\omega'} \mu[\mathbf{A}_1(\omega), \dots, \mathbf{A}_{i-1}(\omega), \mathbf{A}_i(\omega'), \bbE(\mathbf{A}_{i+1}), \dots, \bbE(\mathbf{A}_m)]\\
	&=\mu[\mathbf{A}_1(\omega), \dots, \mathbf{A}_{i-1}(\omega),\sum_{\omega'\in\Omega}c_{\omega'} \mathbf{A}_i(\omega'), \bbE(\mathbf{A}_{i+1}), \dots, \bbE(\mathbf{A}_m)].
	\end{split}
	\]
	Since the last polynomial is the mixed characterstic polynomial of positive matrices, it has all roots real by Theorem~\ref{th:mixed has real roots}. It follows by Theorem~\ref{th:nice families} that $ \{ f_{\omega'}:\omega'\in\Omega \} $ is a nice family. Moreover, if we take as coefficients of the convex combination $ c_{\omega'}=\mathbf{p}(\omega') $ and so for  any $ j=1,\dots, d $ we have $\sum_{\omega'\in\Omega}c_{\omega'} \mathbf{A}_i(\omega')=\bbE(\mathbf{A}_i)  $. Applying the last part of Theorem~\ref{th:nice families} it follows that for any $ j=1,\dots, d $,
	\[
	\begin{split}
&	\min_{\omega'\in\Omega}  \rho_j(\mu[\mathbf{A}_1(\omega), \dots,\mathbf{A}_{i-1}(\omega),
	\mathbf{A}_i(\omega'), \bbE(\mathbf{A}_{i+1}), \dots, \bbE(\mathbf{A}_m)])\\
&\qquad\qquad \rho_j(\mu[\mathbf{A}_1(\omega), \dots,\mathbf{A}_{i-1}(\omega), \bbE(\mathbf{A}_i), \bbE(\mathbf{A}_{i+1}), \dots, \bbE(\mathbf{A}_m)]).	
	\end{split}
	\]
	Taking the minimum with respect to $\omega\in\Omega$, we obtain
		\begin{equation}\label{eq:min min}
		\begin{split}
		&	\min_{\omega\in\Omega} 	\min_{\omega'\in\Omega}  \rho_j(\mu[\mathbf{A}_1(\omega), \dots,\mathbf{A}_{i-1}(\omega),
		\mathbf{A}_i(\omega'), \bbE(\mathbf{A}_{i+1}), \dots, \bbE(\mathbf{A}_m)])\\
		&\qquad\qquad	\min_{\omega\in\Omega}  \rho_j(\mu[\mathbf{A}_1(\omega), \dots,\mathbf{A}_{i-1}(\omega), \bbE(\mathbf{A}_i), \bbE(\mathbf{A}_{i+1}), \dots, \bbE(\mathbf{A}_m)]).	
		\end{split}
		\end{equation}
	Suppose the minimum in the left hand side is attained in $\omega=\omega_0$, $\omega'=\omega'_0$. By independence of the random matrices $\mathbf{A}_i$, we have
	\[
	\begin{split}
	&	\mathbf{p}(\{\sigma\in\Omega : \mathbf{A}_1(\sigma)=\mathbf{A}_1(\omega_0),\dots, \mathbf{A}_{i-1}(\sigma)=\mathbf{A}_{i-1}(\omega_0), \mathbf{A}_i(\sigma)=\mathbf{A}_(\omega'_0)  \})\\
	&\qquad \mathbf{p}(\{\sigma\in\Omega : \mathbf{A}_1(\sigma)=\mathbf{A}_1(\omega_0),\dots, \mathbf{A}_{i-1}(\sigma)=\mathbf{A}_{i-1}(\omega_0)\}) \mathbf{p}(\{\sigma\in\Omega :\mathbf{A}_i(\sigma)=\mathbf{A}_(\omega'_0)  \})>0.
	\end{split}
	\]
Taking $\sigma_0\in \Omega$ in the set in the left hand side, we obtain
\[
\begin{split}
&\min_{\sigma\in\Omega}\rho_j(\mu[\mathbf{A}_1(\sigma), \dots,\mathbf{A}_{i-1}(\sigma),
\mathbf{A}_i(\sigma), \bbE(\mathbf{A}_{i+1}), \dots, \bbE(\mathbf{A}_m)])\\
&\qquad\qquad\le \rho_j(\mu[\mathbf{A}_1(\sigma_0), \dots,\mathbf{A}_{i-1}(\sigma_0),
\mathbf{A}_i(\sigma_0), \bbE(\mathbf{A}_{i+1}), \dots, \bbE(\mathbf{A}_m)])\\
&\qquad\qquad=\rho_j(\mu[\mathbf{A}_1(\omega_0), \dots,\mathbf{A}_{i-1}(\omega_0),
\mathbf{A}_i(\omega'_0), \bbE(\mathbf{A}_{i+1}), \dots, \bbE(\mathbf{A}_m)])\\
&\qquad\qquad =\min_{\omega\in\Omega} 	\min_{\omega'\in\Omega}  \rho_j(\mu[\mathbf{A}_1(\omega), \dots,\mathbf{A}_{i-1}(\omega),
\mathbf{A}_i(\omega'), \bbE(\mathbf{A}_{i+1}), \dots, \bbE(\mathbf{A}_m)]).
\end{split}
\]	
This inequality, together with~\eqref{eq:min min}, implies~\eqref{eq:going from i-1 to i}, finishing thus the proof of the theorem.	
\end{proof}

\begin{remark}\label{re:E(A) is not of rank one}
	The point of Theorem~\ref{th:min root le mixed char of expectation} is that the middle term might be easier to compute or to estimate. But, since the matrices $\bbE(\mathbf{A}_1), \dots, \bbE(\mathbf{A}_m)$ are \emph{not} of rank one, Theorem~\ref{th:charact=mixed char} does not apply, and $\mu[\bbE(\mathbf{A}_1), \dots, \bbE(\mathbf{A}_m)]$ is not a characteristic polynomial. However, Theorem~\ref{th:min root le mixed char of expectation} tells us that its roots can be used to estimate the eigenvalues of $\mathbf{A}(\omega)$ for at least \emph{some} value of $\omega$. 
\end{remark}

\begin{corollary}\label{co:norm bounded for random matrices}
	Let $ \mathbf{A}_1(\omega), \dots, \mathbf{A}_m(\omega) $ be independent rank one positive random matrices in $M_d(\bbC)$, and $\mathbf{A}=\mathbf{A}_1+\dots+ \mathbf{A}_m$. Suppose $\bbE(\mathbf{A})=I_d$ and $ \bbE(\tr \mathbf{A}_i)\le \epsilon $ for some $ \epsilon>0 $. Then
	\[
	\min_{\omega\in\Omega}\|\mathbf{A}(\omega)\|\le (1+\sqrt{\epsilon})^{2}.
	\]
\end{corollary}

\begin{proof}
	Since $\tr(\bbE(\mathbf{A}_i))= \bbE(\tr \mathbf{A}_i)\le \epsilon $, the matrices $\bbE(\mathbf{A}_1), \dots, \bbE(\mathbf{A}_m)$ satisfy the hypotheses of Theorem~\ref{th:sum=1, trace small}, all roots of $ \mu[\bbE(\mathbf{A}_1), \dots, \bbE(\mathbf{A}_m)] $ are smaller than $ (1+\sqrt{\epsilon})^{2} $. By Theorem~\ref{th:min root le mixed char of expectation} we obtain, in particular,
	\[
	\min_{\omega\in\Omega}\|\mathbf{A}(\omega)\|= \min_{\omega\in\Omega} \rho_1(p_{\mathbf{A}(\omega)}) 
	\le \rho_1(\mu[\bbE(\mathbf{A}_1), \dots, \bbE(\mathbf{A}_m)])\le (1+\sqrt{\epsilon})^{2}.
	\qedhere
	\]
\end{proof}

\subsection{Probability and partitions}

The last theorem of this section gets us  closer to the paving conjecture. It is here that we make the connection between the probability space and the partitions.
Let us first note that, similarly to Remark~\ref{re:operators and matrices}, one can see that the independence condition is not affected by a change of basis. So in Theorem~\ref{th:min root le mixed char of expectation} and in Corollary~\ref{co:norm bounded for random matrices} we may assume that $\mathbf{A_i}$ take values in $\LL(V)$ for some finite dimensional vector space $V$.  This observation will be used in the proof of the next theorem.

\begin{theorem}\label{th:first theorem with partitions}
	Suppose $ A_1,\dots, A_m\in M_d(\bbC) $ are positive rank one matrices, such that $\sum_{i=1}^{m} A_i=I_d$ and $\|A_i\|\le C$ for all $i=1,\dots, m$. Then for every positive integer $r$ there exists a partition $S_1,\dots, S_r$ of $ \{1, \dots, m \} $, such that
	\[
	\Big\| \sum\limits_{i\in S_j} A_i\Big\| \le \left( \frac{1}{\sqrt{r}}+\sqrt{C}\right)^2
	\]
	for any $j=1, \dots, r$.
\end{theorem}

\begin{proof}
	Since the purpose is to find a partition with certain properties, we will take as a random space $\Omega$ precisely the space of all partitions of $ \{1, \dots, m \} $ in $r$ sets, with uniform probability $\bf p$. Such a partition is determined by an element $\omega=(\omega_1,\dots, \omega_m)$, where $\omega_j\in\{1, \dots, r \}$, and $S_j=\{k: \omega_k=j\}$; so $\Omega=\{1, \dots, r \}^m$. Also, the different coordinates, that is the maps $\omega\mapsto \omega_i$, are   independent scalar random variables on $\Omega$.

	We consider the space $V:=\bbC^d\oplus\cdots\oplus \bbC^d$ and define the random matrices $\mathbf{A}_i$ ($i=1,\dots, m$) with values in $\LL(V)$ by 
	\begin{equation}\label{eq:decomposition tilde A(omega)}
	\mathbf{A}_i(\omega)=0\oplus\cdots\oplus rA_i\oplus\dots\oplus0,
	\end{equation}
	where $rA_i$ appears in  position $\omega_i$. 
	
	These are independent random matrices (since the coordinates $\omega_i$ are independent). If we fix $1\le j\le r$, then  $\omega_i=j$ with probability $1/r$, and so $rA_i$ appears in  position $j$ with probability $1/r$. Therefore
	\[
		\bbE(\mathbf{A}_i)= \frac{1}{r}rA_i\oplus\cdots \oplus  \frac{1}{r}rA_i= A_i\oplus A_i\oplus\cdots\oplus A_i.
	\]
	If $\mathbf{A}=\mathbf{A}_1+\dots+\mathbf{A}_m$, then 
	\[
	\bbE(\mathbf{A})=  \sum_{i=1}^{m} \bbE (\mathbf{A}_i=
	\sum_{i=1}^{m} (A_i\oplus A_i\oplus\cdots\oplus A_i)=I_V.
	\]
	Since $\tr\mathbf{A}_i(\omega)=r\tr A_i$ for all $i$, we have
	\[
	\bbE(\tr \mathbf{A}_i)= \bbE(r\tr A_i)=r\bbE()\|A_i\|)\le rC.
	\]
	Corollary~\ref{co:norm bounded for random matrices} yields the existence of $ \omega\in\Omega $ such that
	\[
	\|\mathbf{A}(\omega)\|\le (1+\sqrt{rC})^2.
	\]
	But, according to~\eqref{eq:decomposition tilde A(omega)}, we have
	\[
	\mathbf{A}(\omega)=\left(r
	\sum_{\omega_i=1} A_i\right)\oplus \left(r\sum_{\omega_i=2} A_i\right)\oplus\cdots \oplus\left( 	r\sum_{\omega_i=r} A_i.
	\right)
	\]
	We define then $S_j=\{i:\omega_i=j	\}$. It follows that $ \|r\sum_{i\in S_j}A_i\|\le (1+\sqrt{rC})^2 $ for all $j$, and dividing by $r$ ends the proof of the theorem.
\end{proof}

\section{Proof of the Paving Conjecture}\label{se:main proof}

We may now proceed to the proof of the paving conjecture; from this point on all we need from the previous sections is Theorem~\ref{th:first theorem with partitions}. We first deal with orthogonal projections. For such operators the paving conjecture  is trivially verified (exercise: if $P$ is an orthogonal projection and $\diag P=0$, then $P=0$). But we will prove a  quantitative version of the paving conjecture, in which one does not assume zero diagonal. 

\begin{lemma}\label{le:paving for projections}
Suppose $P\in M_m(\bbC)$ is an orthogonal projection. For any $r\in \bbN$ there exists diagonal orthogonal projections $Q_1,\dots, Q_r\in M_m(\bbC)$, with  $\sum_{j=1}^{r}Q_j=I_m$, such that
\[
\|Q_jPQ_j\|\le \left( \frac{1}{\sqrt{r}} + \sqrt{\|\diag P\|} \right)^2
\] 	
for all $j=1,\dots, r$.
\end{lemma}

\begin{proof}
	Denote by $ V $ the image of $P$, and $ d=\dim V $. Let $(e_i)_{i=1}^m$ a basis in $\bbC^m$, and define on $V$  the rank one positive operators $ A_i$ by $A_i(v)=\<v,P(e_i)\> P(e_i) $.
	We have
	\begin{equation}\label{eq:paving proj - norm of A_i}
	\|A_i\|\le \|P(e_i)\|^2=\<  P(e_i), e_i\>\le \|\diag P\|,
	\end{equation}
	and, for $v\in V$,
	\begin{equation}\label{eq:<Av,v>}
	\<A_i v, v\>=\<v,P(e_i)\> \<P(e_i) , v\>=|\<v, Pe_i\>|^2=|\<v,e_i\>|^2.
	\end{equation}
Consequently,
	\[
	\<\sum_{i=1}^{m}A_i v, v\> = \sum_{i=1}^{m}\<A_iv,v\>
 = \sum_{i=1}^{m}|\< v, e_i\>|^2
	=\|v\|^2,
	\]
	whence
	\begin{equation}\label{eq:paving proj - sum of A_i}
	\sum_{i=1}^{r} A_i=I_V.
	\end{equation}
	
	From~\eqref{eq:paving proj - norm of A_i} and~\eqref{eq:paving proj - sum of A_i} it follows that $A_i$ satisfy the hypotheses of Theorem~\ref{th:first theorem with partitions}, with $C=\|\diag P\|_\infty$. There exists therefore a partition $S_1,\dots, S_r$ of $ \{1, \dots, m \} $, such that
	\[
	\Big\| \sum\limits_{i\in S_j} A_i\Big\| \le \left( \frac{1}{\sqrt{r}}+\sqrt{\|\diag P\|_\infty}\right)^2
	\]
	for any $j=1, \dots, r$.
	
	Define then $ Q_j\in M_m(\bbC) $ to be the diagonal orthogonal projection on the span of  $\{e_i  : i\in S_j\}$. Then 
	\[
	\|Q_jPQ_j\|=\|Q_jP(Q_jP)^*\|= \|Q_jP\|^2=\|Q|V\|^2.
	\]
	But, if $v\in V$, then, applying~\eqref{eq:<Av,v>},
	\[
	\|Q_jv\|^2=\sum_{i\in S_j} |\< v, e_i\>|^2=
 \sum_{i\in S_j} \<A_i v, v\> 
	=\<\big(\sum_{i\in S_j} A_i\big) v, v\>
	\le 
	\|\sum_{i\in S_j}A_i\|\cdot \|v\|^2.
	\]
	So
	\[
	\|Q_jPQ_j\|=\|Q|V\|^2\le \|\sum_{i\in S_j}A_i\| \le \left( \frac{1}{\sqrt{r}}+\sqrt{\|\diag P\|_\infty}\right)^2,
	\]
	and the lemma is proved.
\end{proof}

%To get beyond projections, we will use a simple dilation lemma.
%
%\begin{lemma}\label{le:dilation to projections}
%	Suppose $A\in M_d(\bbC)$, and $0\le A\le I$. For any $\epsilon>0$ there exists $d'\in \bbN$ and an orthogonal projection $P\in M_{d+d'}(\bbC)$, such that with respect to the decomposition $\bbC^{d+d'}=\bbC^d\oplus \bbC^{d'}$ we have
%	\[
%	P=\begin{pmatrix}
%	A & *\\
%	* & A'
%	\end{pmatrix}
%	\]
%	with $\|\diag A'\|\le \epsilon$.
%\end{lemma}
%
%
%\begin{proof}
%	Define $\pi_N\in M_N(\bbC)$ to have all entries equal to $1/N$. Then $\pi_N$ is an orthogonal projection (of rank 1). On the other hand, one checks easily that
%	\[
%	P':=\begin{pmatrix}
%	A & A^{1/2} (I-A)^{1/2} \\
%	A^{1/2} (I-A)^{1/2} & I-A
%	\end{pmatrix}
%	\]
%	is an orthogonal projection in $M_{2d}(\bbC)$. If $1/N\le \epsilon$, then one can take $d'=2d+N$ and $P=P'\otimes\pi_N$.
%\end{proof}

\begin{theorem}[The paving conjecture]\label{th:paving conjecture}
	For any $ \epsilon>0 $ there exists $ r\in \bbN $ such that, for any $m\in \bbN$ and $ T\in M_m(\bbC) $ with $\diag T=0$ there exist diagonal orthogonal projections $Q_1,\dots, Q_r\in M_m(\bbC)$, with  $\sum_{j=1}^{r}Q_j=I_m$, such that
	\[
	\|Q_jTQ_j\|\le \epsilon\|T\|
	\] 	
	for all $j=1,\dots, r$.
\end{theorem}

\begin{proof}
	Suppose first that $T=T^*$ and $\|T\|\le 1$. The $2m\times 2m$ matrix
	\[
	P=\begin{pmatrix}
	\frac{I_m+T}{2} & \frac{1}{2}(I_m-T^2)^{1/2}\\
	\frac{1}{2}(I_m-T^2)^{1/2} & \frac{I_m-T}{2}
	\end{pmatrix}
	\]
	 is an orthogonal projection and $\diag P=\left( \frac{1}{2}, \dots,  \frac{1}{2} \right)$.
	 Choose $r$ large enough to have $2 \left(\frac{1}{\sqrt{r}}+\frac{1}{\sqrt{2}}\right)^2-1\le \epsilon$.
	 It follows from Lemma~\ref{le:paving for projections} that there exist diagonal projections $Q_1'', \dots, Q_r''\in M_{2m}(\bbC)$ with 
	 $
	 \sum_{i=1}^{r} Q''_i=I_{2d}$ and $\|Q_i''PQ''_i\|\le\left(\frac{1}{\sqrt{r}}+\frac{1}{\sqrt{2}}\right)^2 $ for all $i=1, \dots, r$.
	 
	 Let $Q_i''=Q_i+Q'_i$ be the decomposition of $Q''_i$ in the diagonal projections corresponding to the first $m$ and the last $m$ vectors of the basis of $\bbC^{2m}$. So $\sum_{i=1}^{r} Q_i =\sum_{i=1}^{r} Q_i=I_m $ and, for each $i=1,\dots, m$,
	 \begin{equation}\label{eq:paving s.a., zero}
	 \|Q_i(I+T)Q_i\| \le 2 \left(\frac{1}{\sqrt{r}}+\frac{1}{\sqrt{2}}\right)^2, \qquad \|Q'_i(I-T)Q'_i\| \le 2 \left(\frac{1}{\sqrt{r}}+\frac{1}{\sqrt{2}}\right)^2.
	 \end{equation}
	 
	 The first inequality implies that $Q_i(I+T)Q_i \le 2 \left(\frac{1}{\sqrt{r}}+\frac{1}{\sqrt{2}}\right)^2Q_i$, so
	 \begin{equation}\label{eq:paving s.a., first}
	 -Q_i\le Q_iTQ_i \le \left[ 2 \left(\frac{1}{\sqrt{r}}+\frac{1}{\sqrt{2}}\right)^2-1 \right] Q_i\le \epsilon Q_i
	 \end{equation}
	(the left inequality being obvious). Similarly, the second inequality in~\eqref{eq:paving s.a., zero} yields
	\begin{equation}\label{eq:paving s.a., second}
	-\epsilon Q_i'\le \left[1- 2 \left(\frac{1}{\sqrt{r}}+\frac{1}{\sqrt{2}}\right)^2 \right] Q'_i\le Q'_iTQ'_i\le Q'_i.
	\end{equation}
	 If we define $Q_{ij}=Q_iQ'_j$ ($i,j=1,\dots, r$), then $\sum_{i,j=1}^{r} Q_{ij}=I_m$, and it follows from~\eqref{eq:paving s.a., first} and~\eqref{eq:paving s.a., second} that
	 \[
	 -\epsilon Q_{ij}\le Q_{ij}TQ_{ij}\le \epsilon Q_{ij},
	 \]
	 or $Q_{ij}TQ_{ij}\le \epsilon$. The theorem is thus proved for $T$ a selfadjoint contraction, and it is immediate to extend it to arbitrary selfadjoint matrices.

If we take now an arbitrary $T\in M_m(\bbC)$, with $\diag T=0$, we may write it as $T=A+iB$, with $A,B$ selfadjoint, $\|A\|, \|B\|\le \|T\|$, and $\diag A=\diag B=0$. Applying the first step, one finds diagonal projections $Q'_1,\dots, Q'_r, Q''_1,\dots, Q''_r\in M_m(\bbC)$, with $\sum_{i=1}^{r} Q'_i=\sum_{i=1}^{r} Q''_i=I_m$, $\|Q_i'AQ_i'\|\le \frac{\epsilon}{2}\|T\|$  and $\|Q_i''BQ_i''\|\le \frac{\epsilon}{2}\|T\|$ for $i=1,\dots, r$. If we define $Q_{ij}=Q_i'Q_j''$, then $\sum_{i,j=1}^{r} Q_{i,j}=I_m$, and 
$\|Q_{ij}TQ_{ij}\|\le \epsilon\|T\|$ for $i,j=1,\dots, r$.
\end{proof}

By writing carefully the estimates in the proof, one sees also that we may take $r$ of order $\epsilon^{-4}$.

\section{Final Remarks}\label{se:final remarks}

1. 
As noted in Remark~\ref{re:dirac}, there is a connection between the Kadison--Singer problem and quantum mechanics. We will give here a very perfunctory account. In the von Neumann picture of quantum mechanics, states (in the common sense) of a system correspond to states $\phi$ (in the $C^*$-algebra sense) of $B(\HH)$, while observables of the system correspond to selfadjoint operators $A\in B(\HH)$. The value of an observable in a state is precisely $\phi(A)$. 

A maximal abelian $C^*$-algebra $\AA\subset B(\HH)$ corresponds to a maximal set of mutually compatible observables. If the extension of any pure state on $\AA$ to a state on $B(\HH)$ is unique, then one can say that the given set of observables determines completely all other observables. This seems to have been assumed by Dirac implicitely.

Now, there are various  maximal abelian subalgebras of $B(\HH)$, but the problem can  essentially be reduced to two different basic types: continuous (that are essentially isomorphic to $L^\infty$ acting as multiplication operators on $L^2$) and discrete (that are isomorphic to $\DD$ acting in $\ell^2$). The main topic of the original paper~\cite{KS} is to prove that extension of pure states is \emph{not} unique in general for continuous subalgebras. They suspected that the same thing happens for the discrete case, but could not prove it, and so posed it as an open problem.

\smallskip
2.
We have said in the introduction that there are many statements that had been shown to be equivalent to (KS), besides (PC) that we have used in an essential way. We have thus, among others:

\begin{enumerate}
	\item Weaver's conjectures in discrepancy theory. The original proof in~\cite{MSS} goes actually through one of these; the shortcut using  (PC) is due to Tao~\cite{T}.
	
	\item Feichtinger's conjecture in frame theory.

	\item Bourgain--Tzafriri conjecture.
\end{enumerate}

All these conjectures have in fact different forms, weaker or stronger variants, etc---a detailed account may be found in~\cite{Ca}. It is worth noting that up to 2013 most specialists believed that they are not true, and that a counterexample will eventually be found. So it was a surprise when all these statements were simultaneously shown true by~\cite{MSS}.

\smallskip
3.
The method used in~\cite{MSS} is even stronger than described above. Actually, its first application was to a completely different problem in graph theory: the existence of certain infinite families of so-called Ramanujan graphs~\cite{MSS2} (see also~\cite{MSS3} for an account).

\smallskip
4.
The most tedious proof in the above notes is that of Lemma~\ref{le:basic conditions on Phi}. The original argument in~\cite{MSS} is more elegant, but uses another result of Borcea and Br\"and\'en~\cite{BB2} that represents real stable polynomials in two variables as determinants of certain matrices---a kind of converse to Lemma~\ref{le:example of real stable}. The direct argument we use appears in~\cite{T}.

\end{document}